\documentclass[10pt]{amsart}
\usepackage{amsmath,amsfonts,amsthm,amssymb,amscd}
\def\classification#1{\def\@class{#1}}
\classification{\null}

\DeclareFontFamily{OT1}{rsfs}{}
\DeclareFontShape{OT1}{rsfs}{n}{it}{<-> rsfs10}{}
\DeclareMathAlphabet{\mathscr}{OT1}{rsfs}{n}{it}

\newcommand{\eps}{\varepsilon}

\newcommand{\R}{{\mathbb R}}
\newcommand{\Pro}{{\mathbb P}}

\newcommand{\Z}{\mathbb{Z}}

\newcommand{\F}{\mathbb{F}}
\newcommand{\E}{\mathbb{E}}

\newcommand{\Ka}{\mathcal{K}}

\def\G{\Gamma}
\def\E{\mathsf {E}}

\newtheorem{theorem}{Theorem}
\newtheorem{proposition}[theorem]{Proposition}

\newtheorem{corollary}[theorem]{Corollary}
\theoremstyle{remark}
\newtheorem{remark}[theorem]{Remark}

\title{New sum-product type estimates over finite fields}

\author{Oliver Roche-Newton, Misha Rudnev and Ilya D. Shkredov}

\address{O. Roche-Newton: School of Mathematics and Statistics, Wuhan University, Wuhan, Hubei Province, P.R.China. 430072 }
\email{o.rochenewton@gmail.com }

\address{Misha Rudnev, Department of Mathematics, University of Bristol,
Bristol BS8 1TW, United Kingdom}
\email{m.rudnev@bristol.ac.uk}
\address{Ilya D. Shkredov, Steklov Mathematical Institute, Division of Algebra and Number Theory, ul. Gubkina, 8, Moscow, Russia, 119991 and IITP RAS, Bolshoy Karetny per. 19, Moscow, Russia, 127994}
\email{ilya.shkredov@gmail.com}

\begin{document}
\begin{abstract} Let $F$ be a field with positive odd characteristic $p$.  We prove a variety of new sum-product type estimates over $F$.  They are derived from the theorem that the number of incidences between $m$ points and $n$ planes in the projective three-space $PG(3,F)$, with $m\geq n=O(p^2)$,
is
$$O( m\sqrt{n} + km ),$$
where $k$ denotes the maximum number of collinear  planes.

The main result is a significant improvement of the state-of-the-art sum-product inequality over fields with positive characteristic, namely that
\begin{equation}\label{mres}
|A\pm A|+|A\cdot A| =\Omega \left(|A|^{1+\frac{1}{5}}\right),
\end{equation} for any $A$ such that $|A|<p^{\frac{5}{8}}.$

\end{abstract}

\maketitle

\section{Introduction}  Let $F$ be a field with positive odd characteristic $p$, i.e. $F=\F_q$, where $q$ is a power of the prime $p$. In this paper we prove combinatorial-geometric estimates on sum and product sets over $F$, which are in a certain sense similar to those over the real and complex fields, obtained geometrically via the Szemer\'edi-Trotter theorem after the work of Elekes, \cite{E}. See also \cite{ENS}, \cite{ER}, \cite{ss}, \cite{LR1}, \cite{KR}. Our new results appear considerably stronger that what has been known so far in the finite field setting, where the main techniques were arithmetic-combinatorial and were among other sources laid down in \cite{BKT}, \cite{BGK}, see also \cite{tv} as a general reference.

For instance, we establish a new sum-product bound
$$
|A\pm A|+|A\cdot A| =\Omega \left(|A|^{1+\frac{1}{5}}\right),
$$ for any $A\subset F$ such that $|A|<p^{\frac{5}{8}}.$ This is a considerable improvement over the previously established best results in \cite{RR}, \cite{LR}, which were based on purely arithmetic techniques. In spirit, our main result is akin to the well-known sum-product estimate of Elekes, \cite{E}, yielding the exponent ${1+\frac{1}{4}}$ for reals.

As usual, we use the notation $|\cdot|$ for cardinalities of finite sets. Symbols $\ll$, $\gg,$ suppress absolute constants in inequalities, as well as respectively do the symbols $O$ and $\Omega$. Besides,  $X=\Theta(Y)$ means that $X=O(Y)$ and $X=\Omega(Y)$. The symbols $C$ and $c$ stand for absolute constants, which may change from line to line. When we turn to sum-products, we use the standard notation
$$A+A=\{a_1+a_2:\,a_1,a_2\in A\}$$ for the sumset $A+A$ of $A\subseteq F$, and similarly for the product set $AA$, alias $A\cdot A$. Sometimes we write $nA$ for multiple sumsets, e.g. $A+A+A=3A$, as well as $A^{-1}=\{a^{-1}: a\in A\setminus\{0\}\}.$

We use in the paper  the same letter to denote a set
$S\subseteq F$ and its characteristic function $S:F\rightarrow \{0,1\}.$
We write $\E(A,B)$ for the {\it additive energy} of two sets $A,B \subseteq F$, that is
$$
    \E(A,B) = |\{ a_1+b_1 = a_2+b_2 ~:~ a_1,a_2 \in A,\, b_1,b_2 \in B \}| \,.
$$
If $A=B$ we simply write $\E(A)$ instead of $\E(A,A).$ Similarly, $$\E_k (A) = |\{ a_1-a'_1 = \dots = a_k-a'_k ~:~ a_j,a'_j \in A \}|.$$

Throughout the paper $P$ will denote a set of $m$ points in $F^3$ or  $PG(3,F)$ and $\Pi$ a set of  $n$ planes.

Given an arrangement $\{P,\Pi\}$ of planes and points in $F^3$, the set of incidences is defined as
$$
I(P,\Pi) =\{(\rho,\pi)\in P\times \Pi:\,\rho\in\pi\}.
$$

\medskip
The main tool in this paper is an incidence theorem proven by the second author, \cite{R}, as follows.

\begin{theorem}  \label{mish} Let $P, \Pi$ be sets of points and planes, of cardinalities respectively $m$ and $n$, in  $PG(3,F)$, with $m\geq n$.  Suppose, $F$ has positive characteristic $p\neq 2$ and $n=O(p^2)$. Let $k$ be the maximum number of collinear planes.

Then
\begin{equation}\label{pups}
|I(P,\Pi)|=O( m\sqrt{n} + km).\end{equation}
\end{theorem}
The assumption $m\geq n$ can be reversed in an obvious way using duality. In fact, for our applications one has roughly $m=n$.
Note that the trivial $km$ term may dominate the estimate \eqref{pups} only if $k=\Omega(\sqrt{n})$.

\begin{remark}On the technical level, to prove Theorem \ref{mish}, it suffices to adapt the algebraic proof of Theorem 2.10 in \cite{GK}. This enables one to bypass the polynomial partitioning technique, which relies on the order properties of reals, and hence, at least in its present form, does not extend beyond the real case.

The geometric concept enabling such a conversion in \cite{R} was to interpret incidences between points and planes  in the ``physical'' projective three-space  $PG(3,F)$ in the ``phase space'' of lines in $PG(3,F)$, namely the Klein quadric $\mathcal K\subset PG(5,F)$. See e.g. \cite{JS} for theoretical foundations. In short, unless $F$ has characteristic $2$, $\mathcal K$ has two rulings by two-planes, called $\alpha$- and $\beta$-planes. A point in $PG(3,F)$ gives rise to a two-plane in $\mathcal K$, a so called $\alpha$-plane, being the Klein image of the set of lines incident to the above point in the physical space. A plane in $PG(3,F)$ corresponds in $\mathcal K$ to a two-plane, a so called $\beta$-plane, the Klein image of the set of lines lying in the above plane in the physical space.  Two distinct planes of the same type always intersect in $\mathcal K$ at a point. Two planes of different types intersect in $\mathcal K$  if and only if the corresponding point and plane in $PG(3,F)$ are incident to one other. In this case, the corresponding $\alpha$ and $\beta$-plane in $\Ka$ then intersect degenerately, along a line. Thus one can restrict the corresponding incidence problem in $\mathcal K$ to the transverse intersection $\mathcal G$ of $\mathcal K$ with a random subspace $PG(4,F)$ in $PG(5,F)$. $\mathcal G$ is called a non-degenerate line complex. Replacing $F$ with its algebraic closure, the random non-degenerate line complex $\mathcal G$ may be chosen so that it does not contain any of the finite number of points in $\mathcal K$, where planes of the same type intersect. This is what makes it possible then to proceed with the proof of Theorem \ref{mish} along the lines of Theorem 2.10 in \cite{GK}.

But if $F$ has characteristic $2$, the whole paradigm seems to break down, for the two rulings of $\mathcal K$ with $\alpha$- and $\beta$-planes (arising respectively from the sets of orthogonal $3\times 3$ matrices with determinants $\pm 1$) coincide. Hence, even though Theorem \ref{mish} is vacuous for the small characteristic $p$, we have chosen to explicitly exclude the case $p=2$. \end{remark}

Some applications of Theorem \ref{mish} were pointed out in \cite{R}, here we aim to extend their scope, combining the estimate (\ref{pups}) with other tools that have been developed in the area. Note that Theorem \ref{mish} is formulated in a way so that it applies to all fields with characteristic $p$, in particular the prime residue field $\F_p$.  Thus $p$ inevitably appears in the hypotheses of the theorem, and hence a set $A\subset F$ for which we develop sum-product type estimates cannot be too large in terms of $p$. It is likely that under additional assumptions, $p$ can be replaced by $q$ both in Theorem \ref{mish} and its applications herein in the general finite field case $\F_q$.  This essentially comes down to removing the restriction in terms of $p$ in the classical Monge proof of the Cayley-Salmon theorem about flechnodal points on surfaces, underlying the proof of Theorem \ref{mish}. In general, this cannot be done, for there are counterexamples. However, one may conjecture an {\em only if} condition as to the scope of these counterexamples. See \cite{Vo}, \cite{Ko} and the references contained therein for more discussion in this direction.

\section{Acknowledgements} Part of this research was performed while the authors were visiting the Institute for Pure and Applied Mathematics (IPAM), which is supported by the National Science Foundation.

The first author was supported by the Austrian Science Fund (FWF): Project F5511-N26, which is part of the Special Research Program ``Quasi-Monte Carlo Methods: Theory and Applications. The third author was supported by grant Russian Scientific  Foundation RSF 14-11-00433.

The authors thank Adam Sheffer and Mubariz Garaev for comments throughout the
preparation of this manuscript. Special thanks to Igor Shparlinski for pointing out some applications of Theorem \ref{t:g^j}.

\section{Sum-product type estimates}
This section contains the short proof of our main result. It is  followed by a series of remarks, placing it in the context of the current state of the art. In a separate subsection we show how the exponents can be improved slightly, once longer sum sets have been taken.

\medskip
Theorem \ref{mish} enables one to count the maximum number of solutions of bilinear equations with six variables in some discrete sets.

\begin{theorem}\label{tool}
Let $A,B,C\subseteq F$, let $M = \max(|A|,|B|,|C|).$ Then
$$
|A+BC| =\Omega\left[\min\left(\sqrt{|A||B||C|},\frac{1}{M}|A||B||C|,\, p\right)\right].
$$\end{theorem}
\begin{proof}

Suppose, $|A||B||C|\leq cp^2$, with $c$ implicit in Theorem \ref{mish}. If this condition is not satisfied, pass to subsets $A,B,C$ to ensure it is just satisfied, so that the first and the third term in the claimed estimate are of the same order of magnitude.

Let $E$ be the number of solutions of the equation
\begin{equation}\label{threeen}
a+bc = a'+b'c',\qquad (a,b,c,a',b',c')\in A\times B\times C \times A\times B\times C.
\end{equation}
Let $P,\Pi$ be sets of points and planes in $F^3$, as follows:

$$\begin{aligned}
P & =\{(a,c,b')\},\\
\Pi & = \{\pi:\,x+by - c'z = a'\}.
 \end{aligned}
$$
Hence, there are $m=|A||B||C|$ points and $m$ planes. The maximum number $k$ of collinear points or planes is $M$.

Applying Theorem \ref{mish} yields
\begin{equation}
E=O\left(m^{\frac{3}{2}} + M m \right).
\label{energy}\end{equation}
By the Cauchy-Schwarz inequality,
$$
|A+BC|\geq \frac{|A|^2|B|^2|C|^2}{E}.
$$
The claim of the theorem follows, the characteristic $p$ appearing therein in the case when $|A||B||C|\gg cp^2$ .
\end{proof}
\begin{corollary}Let $A\subseteq F$. Then, for any $a\neq 0$,
$$
|aA\pm AA|=\Omega \left[ \min (|A|^{\frac{3}{2}},p) \right ], \qquad |A\pm AA \pm AA|=\Omega \left [ \min(|A|^{\frac{7}{4}},p)\right ].
$$
Suppose, in addition, that $|AA|=K|A|$, for some $K\geq 1$. Then
$$
|AA\pm AA|=\Omega \left [ \min (K^\frac{1}{2}|A|^{\frac{3}{2}},p) \right ], \qquad |AA+AA \pm AA|=\Omega \left[ \min (K^\frac{1}{4}|A|^{\frac{7}{4}},p) \right ].
$$
\label{aux}
\end{corollary}
\begin{proof}
The estimates  follow by applying Theorem \ref{tool}, with $B,C=A$ and $A=aA,\,, A\pm AA,\,AA, AA\pm AA,$ respectively.
\end{proof}

\begin{remark} If $A$ is a set of positive reals, an elementary order-based argument  shown to us by M. Garaev proves that for some $a,b,c\in A$, $|aA+(b-c)A|\geq|A|^2.$ Just take $a$ the maximum element of $A$, and $b,c$ such that $|b-c|$ is the smallest possible and non-zero. This clearly implies that  $|AA+AA-AA|\geq|A|^2$ (but without one being able to replace the left-hand side with $|AA+AA+AA|$). A similar well-known argument gives the lower bound for $|A\pm AA| \geq|A|^2$ in the positive integer case.

This can be compared with the claims of the above Corollary, which works for any field, except being vacuous in the small positive characteristic case. But the above-mentioned elementary arguments do not enable one (as far as we understand) to get a more general energy-type bound similar to \eqref{energy}. This type of bound in the context of the set $AA\pm AA$ was obtained in \cite{R}, via Theorem \ref{mish}.
\end{remark}

Now we can derive a pure sum--product result in $F$.

\begin{theorem}\label{t:sum-prod}
    Let $A,B,C\subseteq F$, let $M = \max(|A|,|BC|).$
    Suppose that $|A||B||BC| \ll p^2$.
    Then
\begin{equation}\label{f:sum-prod_energy}
    \E(A,C) \ll (|A||BC|)^{3/2} |B|^{-1/2} + M |A||BC| |B|^{-1} \,.
\end{equation}
    In particular, for any $A$ with $|A| < p^{5/8}$ the following holds
\begin{equation}\label{f:sum-prod}
    \max\{ |A \pm A|, |AA| \} \gg |A|^{1 + \frac{1}{5}}
        \,,
\end{equation}
    and
\begin{equation}\label{f:sum-prod'}
    \max\{ |A \pm A|, |A:A| \} \gg |A|^{1 + \frac{1}{5}}
        \,.
\end{equation}
\end{theorem}
\begin{proof}
The main estimate for the purpose of this proof is  for the energy
$$
\E(A,C) =  |\{ a+c = a'+c' ~:~ a,a'\in A,\, c,c' \in C \}|.
$$
    We may clearly assume that $0\not\in B$ and that the sets are non-empty. Then
 \begin{equation}\begin{aligned}
  \E(A,C) & =  |B|^{-2} |\{ a+cb/b = a'+c'b'/b' ~:~ a,a'\in A,\, c,c' \in C,\,b,b'\in B \}| \\
  &\leq |B|^{-2} |\{ a+ st  = a'+s't' ~:~ a,a'\in A,\, s,s' \in BC,\,t,t'\in B^{-1} \}|.\end{aligned}
  \label{rear}\end{equation}

  We now apply  the key estimate \eqref{threeen} in the proof of Theorem \ref{tool} with $A=A$, $B=BC$, $C=B^{-1}$. It follows that
$$
    |B|^2 \E(A,C) \ll m^{3/2} + mM,
$$
with $m=|A||B||BC|$ and $M=\max(|A|,|B|,|C|).$
This proves (\ref{f:sum-prod_energy}).

Moving on to proving \eqref{f:sum-prod}, we assume $|AA|\ll |A|^{\frac{6}{5}},$ or there is nothing to prove. We therefore use \eqref{f:sum-prod_energy} with $A=B=C=A$. Then the condition $|A| < p^{5/8}$ guarantees that $|A|^2 |AA| \ll p^2$, and therefore we can apply the the bound (\ref{f:sum-prod_energy}), and the first term dominates therein.

Then, by the Cauchy--Schwarz inequality, we obtain
$$
    \frac{|A|^4}{|A\pm A|} \le \E(A) \le |A| |AA|^{3/2},
$$
which gives

\begin{equation}
\label{SP}
|A\pm A|^2|AA|^3 \gg |A|^6,
\end{equation}
and we are done.
The bound (\ref{f:sum-prod'}) is established in exactly the same way, but instead taking $B=A^{-1}$ in the application of \eqref{f:sum-prod_energy}. This concludes the proof.
\end{proof}

\medskip
We make a few observations, putting our bounds in the context of other results concerning various sum-product type questions.
\begin{remark} For large enough subsets of the finite field $F=\mathbb F_q$,  that is usually when $|A|\gg \sqrt{q}$, sum-product estimates can be obtained by calculations which may involve exponential sums, see e.g. \cite {G}, but in essence can often be reduced just to linear algebra, see \cite{V}, scarcely making a difference between $\F_q$ and the prime residue field $\F_p$.

So let $F=\mathbb F_p$. It is known that there exists $A$, such that
$$ \max\{ |A \pm A|, |AA| \} \ll \sqrt{p}\sqrt{|A|}.$$
Garaev, \cite{G}, uses exponential sums to reverse the above inequality, provided that $|A|\gg p^{\frac{2}{3}}.$ But he can only do
\begin{equation}\label{Gar}|A\pm A||AA| \gg \frac{|A|^4}{p}\end{equation}
for smaller $|A|$.
In this paper the condition $|A|<p^{5/8}$ arises to ensure that Theorem 1 applies. However, this is precisely when, assuming $|AA|\approx |A\pm A|$, the estimate \eqref{SP}  becomes stronger than \eqref{Gar}.
\end{remark}

\begin{remark} Let us call the {\em weak Erd\H os-Szemer\'edi conjecture} a claim that for a set $A$ in a field $F$ (small enough if $F$ has large positive characteristic $p$), if $|AA|=O(|A|^{1+\delta})$, then $|A+A|=\Omega(|A|^{2-\epsilon})$ where the small parameters $\delta,\epsilon$ are related algebraically. Let us write this statement in shorthand as \begin{equation}\label{wes}|AA|\lesssim |A|\;\;\Rightarrow \;\;|A+A|\gtrsim|A|^2.\end{equation}The question is open for any field but rationals, for it has been resolved in the integer setting by Bourgain and Chang, \cite{BC}. This furthered an earlier paper by Chang, \cite{Ch}, which showed that $|A+A|\gtrsim|A|^2$, given that $|AA|$ does not exceed $|A|$ by a just factor of $\log|A|$. Moreover, Chang then showed that the latter result readily extends over $\R$, owing to the subspace theorem, \cite{Ch1}.
\label{thr}
\end{remark}

\begin{remark} Techniques and concepts engaged to explore the weak Erd\H os-Szemer\'edi conjecture have been quite different from the geometric incidence approach, which is also the scope of this paper. The latter approach has been fruitful to establish quantitative sum-product inequalities in all ranges of $|AA|,|A+A|$.

If $F=\mathbb R$ or $\mathbb C$, Elekes and Elekes-Ruzsa -- see respectively \cite{E}, \cite{ER} -- showed that  the following inequalities can be derived from the Szemer\'edi-Trotter theorem:
\begin{eqnarray} |A\pm A|^2|AA|^2 &  \gg  & |A|^5, \label{el} \\ \hfill \nonumber \\ 
 |A\pm A|^4|AA| \log{|A|} &  \gg  & |A|^6.\label{elr}\end{eqnarray}
Solymosi's  inequality
\begin{equation}\label{si}|A+A|^2|AA|\log{|A|} \gg {|A|^4}\end{equation} also holds over $\mathbb R$ and $\mathbb C$, \cite{solymosi}, \cite{KR}.
Observe that the above-mentioned inequalities \eqref{el}, \eqref{si}, as well as our inequality \eqref{SP} all imply only that
\begin{equation}|AA|\lesssim |A|\;\;\Rightarrow \;\;|A+A|\gtrsim|A|^{3/2}.\label{thri}\end{equation}  The same threshold  exponent $\frac{3}{2}$ in the context of large sets in finite fields is implicit in the constraint $|A|\gg p^{\frac{2}{3}}$ mentioned in the previous remark.

The worst possible case for incidence-based sum-product inequalities is precisely the case when $|AA|$ is relatively small, that is the scope of the weak Erd\H os-Szemer\'edi conjecture. See the second line of estimates in Corollary \ref{aux} in this paper, as well as Section 6.1 in  \cite{R}.  This agrees with Chang's observation in \cite{Ch1} that general geometric incidence theorems of Szemer\'edi-Trotter type appear to be insufficient for the resolution of the weak Erd\H os-Szemer\'edi conjecture. 
\label{thr1}
\end{remark}

\begin{remark} The current state of the art has not yet enabled one to efficiently combine the results and techniques addressed in the above two remarks. 
Still, the inequality \eqref{thri} has been given some improvement  for real and complex fields, owing to the idea of using higher degree convolutions, which has been exploited in a series of works of the third author and collaborators.  See e.g. \cite{ss}, \cite{LR1}, as well as \cite{VS}, where this was done in the context of sums of multiplicative subgroups in $\F_p$.

A recent work \cite{KS} of the third author and S.V. Konyagin sets a new ``world record'', replacing the exponent $\frac{3}{2}$ in the inequality \eqref{thri}  over the real and complex fields by $\frac{19}{12}$. Moreover, over these fields the paper succeeds in establishing the state-of-the-art sum-product inequality
$$|A+A| +|AA| \gg |A|^{4/3+c},$$
for some small $c>0$, having improved the previously best known exponent $\frac{4}{3}$, established by Solymosi, \cite{solymosi}, back in 2008.
\label{thr2}
\end{remark}

\begin{remark} Our sum-product exponent $\frac{6}{5}$ in Theorem \ref{t:sum-prod} coincides with the one established in a different way by Bloom and Jones for the filed $\F_q(x^{-1})$ of Laurent series  over $\F_q$, the implicit constants depending on $q$, as well as the $p$-adic field $\mathbb Q_p$. In principle, our result implies that the above constants can be made $q$-independent, for sets of the size $O(p^2)$. Indeed, one can take for $F$ the algebraic closure of $\F_p$. Then Theorem  \ref{t:sum-prod} extends to the polynomial ring $F[x]$ and the rational field $F(x)$, since a polynomial has a finite number of zeroes and therefore one can find $x=x_0$, such that the evaluation map from $F(x)$ to $F$, acting simply as setting  $x=x_0$ preserves the structure of sums and products. Presumably, the same can be done with semi-infinite Laurent series after suitably truncating them.

We have learnt this basic evaluation argument for how to extend sum-product estimates from a  ring $F$ to $F[x]$ from the paper of Croot and Hart, \cite{CH}. They prove a very strong statement that $|AA|<|A|^{1+c}$ for some $c>0$ implies $|A+A|\gg |A|^2$ for any set $A$ of monic polynomials in the polynomial ring $\mathbb C[x]$, and in principle for any set of polynomials, once no two polynomials are scalar multiples of each other. Their proof for polynomials  extends verbatim for positive integers, conditional on a certain version of the Last Fermat Theorem. \label{blj}\end{remark}

\subsection{Estimates for longer sum sets}

By iteratively applying Theorem \ref{t:sum-prod}, and in particular the inequality \eqref{SP}, it is possible to obtain slightly improved estimates for longer sum sets. Here we confine ourselves to inequalities involving the sets $3A$ and $4A$ versus $AA$. These estimates can be compared with those for convex sets of reals, obtained by iteratively applying the Szemer\'edi-Trotter theorem by Elekes, Nathanson and Ruzsa, \cite{ENS}.

We have:

\begin{corollary} \label{t:sum-prod2}
Let $A \subset \mathbb F_q$ such that $|A|<cp^{18/35}$. Then
\begin{equation}\label{three}|A+A+A|^4|AA|^9 \gg |A|^{16},\end{equation}
and in particular
$$\max \{|A+A+A|,|AA|\} \gg |A|^{16/13}.$$
\end{corollary}

\begin{proof}
Let us assume that $|A+A|^4|AA|^9 \leq |A|^{16}$, as otherwise there is nothing to prove. Therefore,
\begin{equation}
|A+A||A||AA| = (|A+A|^4|AA|^9)^{1/9}|A+A|^{5/9}|A| \leq |A|^{35/9} \leq cp^2.
\label{conditionmet}
\end{equation}

This means that the inequality \eqref{f:sum-prod_energy} can be applied, as well as Cauchy-Schwarz, to deduce that
$$\begin{aligned}
\frac{|A|^2|A+A|^2}{|A+A+A|}&\leq \E(A+A,A)
\\&\ll (|A+A||AA|)^{3/2}|A|^{-1/2}+M|A+A||AA||A|^{-1}
\\&\leq (|A+A||AA|)^{3/2}|A|^{-1/2}+|A+A|^2|AA||A|^{-1}+|A+A||AA|^2|A|^{-1}, 
\end{aligned}$$
where $M=\max \{|A+A|,|AA|\} \leq |A+A|+|AA|$. It is straightforward to check that the first term on the right-hand side of the latter estimate is dominant. Indeed, since $|A+A|^{1/2} \leq (|A||AA|)^{1/2}$, we have
$$|A+A|^2|AA||A|^{-1} \leq |A+A|^{3/2}|AA||A|^{-1}(|A||AA|)^{1/2}=(|A+A||AA|)^{3/2}|A|^{-1/2},$$
and a similar calculation can be made in order to verify that the third term is dominated by the first. It follows that
$$|A+A+A||AA|^{3/2} \gg |A|^{5/2}|A+A|^{1/2}.$$
Finally, since $|A|<cp^{18/35} \leq p^{3/5}$, we can apply \eqref{SP}. This is not quite immediate, since the latter estimate was obtained under the assumption that $|A|<p^{\frac{5}{8}}$ and
$|AA|\ll |A|^{\frac{6}{5}}.$ However,  \eqref{SP} holds regardless of the latter assumption, provided that $|A| \ll p^{3/5}$. Indeed, if $|AA| \gg |A|^{4/3}$ then
$$|AA|^3|A+A|^2 \gg |A|^4|A+A|^2 \geq |A|^6,$$
i.e., \eqref{SP}. On the other hand, if $|AA| \ll |A|^{4/3}$ then, if $|A| \ll p^{3/5}$, one has
$$|A|^2|AA| \ll |A|^{10/3} \ll p^2.$$
But the latter condition was exactly the one to result in \eqref{SP} in the proof of Theorem \ref{t:sum-prod}.

We therefore use \eqref{SP} to get
$$|A+A+A||AA|^{3/2} \gg |A|^{5/2}|A+A|^{1/2} \gg |A|^{5/2} \frac{|A|^{3/2}}{|AA|^{3/4}}.$$
Rearranging this inequality gives
\begin{equation}\label{three2}|A+A+A|^4|AA|^9 \gg |A|^{16},\end{equation}  as required.

\end{proof}

If one continues to iterate this procedure, further small improvements are obtained for longer sum sets. For instance, it follows that if $|A|<cp^{58/101}$ then
$$\max \{|A+A+A+A|,|AA|\} \gg |A|^{36/29}.$$
For longer sum sets, the proofs of such estimates become increasingly long, whilst the gains become increasingly small, and so the details are omitted here.

\begin{remark} Comparing the last two estimates in Corollary \ref{aux} with, respectively, \eqref{SP} and \eqref{three}, one sees that the former two get better and the latter two, in terms of the size of $A+A$ and $A+A+A$, get worse, when the size of the product set $AA$ increases. Corresponding pairs of estimates meet when $|AA|\sim |A|$, i.e. when $A$ is an approximate multiplicative subgroup, in which case, for small enough $A$, one has $|A+A|,\,|AA+AA|\,\gg\,|A|^{\frac{3}{2}}$ and $|3A|,\,|3AA|\,\gg\,|A|^{\frac{7}{4}}$. Exponents in these estimates coincide with those for convex sets of reals in \cite{ENS}. If $A$ is a genuine multiplicative subgroup, the estimates, at least over prime fields, can be improved slightly, see \cite{VS} and \cite{Hart_A+A_subgroups}, by using higher order convolutions and a Stepanov method-based estimate twice, rather than once.  The same trick works for applications of Szemer\'edi-Trotter type bounds for sum-product estimates over the real and complex fields, enabling the improvements in \cite{ss} and \cite{LR1} over the foundational result in \cite{E}.  It is not clear whether Theorem \ref{mish} provides enough flexibility to allow for a more involved application in a similar vein. Even though this is a technical question, the positive answer would enable one to  ``break'' the threshold, in terms of the method's efficiency, exponent $\frac{3}{2}$ for the size of the  sumset of an {\em approximate} multiplicative subgroup in the positive characteristic case. See Remarks \ref{thr}--\ref{thr2}.
\end{remark}

\section{Sets $A(A+A)$ and $(A+A)(A+A)$.}

In the main
result
of the section
we obtain a lower bound for the cardinalities of the
sets $A(A+A)$ and $(A+A)(A+A)$.
Technically, the proofs are based on the fact that we can variate over $B$ in formula (\ref{f:sum-prod_energy}) of Theorem \ref{t:sum-prod} and that the common additive energy $\E(A,A\pm A)$ is known to be large, see \cite{ss}.

Heuristically, the results in this section relate to the converse of the question discussed in Remarks \ref{thr}--\ref{thr2}, that is how large should the product set be, given that the sumset is small. Our main tool is still a version of the inequality \eqref{SP}, which implies that if $|A|$ is small enough relative to the characteristic $p$,
$$
|A+A|\lesssim |A|\;\;\Rightarrow\;\; |AA|\gtrsim |A|^{4/3},
$$ in the notation of Remark \ref{thr}. This is weaker than \eqref{thri} and much weaker than the optimal consequence in this vein of \eqref{elr}, \eqref{si}, holding for real and complex numbers. See also Theorem 9 in \cite{KS}.

\begin{theorem}\label{t:A(A+A)}
    Let $A,B,C\subseteq F$.
    Suppose that $|B||C||(A+C)B| \ll p^2$.
    Then
\begin{equation}\label{f:A(A+A)_1}
    |A||B||C| \ll (|B||C|)^{1/2} |(A+C)B|^{3/2} + |(A+C)B|^2 \,.
\end{equation}
    In particular, for any set $A\subseteq F$, $|A| < p^{3/5}$, we obtain
\begin{equation}\label{f:A(A+A)_2}
    |A(A\pm A)| \gg |A|^{4/3} \,.
\end{equation}
    Further for any $A,B,C,D\subseteq F$ with $|B+D||C||(A+C)(B+D)| \ll p^2$ the following holds
\begin{equation}\label{f:A(A+A)_1'}
    |A||B+D||C| \ll (|B+D||C|)^{1/2} |(A+C)(B+D)|^{3/2} + |(A+C)(B+D)|^2 \,.
\end{equation}
    In particular, for any set $A\subseteq F$, $|A+\eps A|^{4/3} |A|^2 \ll p^{2}$, we
    get
\begin{equation}\label{f:A(A+A)_2'}
    |(A \pm A)(A + \eps A)| \gg |A| |A + \eps A|^{1/3} \,,
\end{equation}
    where $\eps = \{ -1,1 \}$.
\end{theorem}
\begin{proof}
    Let us calculate the common additive energy $\E(A+C,C)$ in two ways.
    On the one hand by Katz--Koester trick (see \cite{kk}), we have for any $s\in C-C$ that
    $$|(A+C) \cap (A+C+s)| \ge |A+C\cap (C+s)| \ge |A|$$
    and hence
\begin{equation}\label{f:13.07.2014_1}
    |A| |C|^2 \le \sum_s |(A+C) \cap (A+C+s)| |C\cap (C+s)| = \E(A+C,C) \,.
\end{equation}
    On the other hand, applying Theorem \ref{t:sum-prod} with $A=C$, $B=B$, and $C=A+C$, we get
\begin{equation}\label{f:13.07.2014_2}
    \E(A+C,C) \ll |C|^{3/2} |(A+C)B|^{3/2} |B|^{-1/2} + |(A+C)B|^2 |C| |B|^{-1} \,.
\end{equation}
    Combining bounds (\ref{f:13.07.2014_1}), (\ref{f:13.07.2014_2}), we obtain
    formula (\ref{f:A(A+A)_1}).
    To get (\ref{f:A(A+A)_2}) just note that if we put in formula (\ref{f:A(A+A)_1})
    that $A=A$, $B=A$, $C=\pm A$ then the first term in (\ref{f:13.07.2014_2}) dominates.
    To prove (\ref{f:A(A+A)_1'}) replace in (\ref{f:A(A+A)_1}) the set $B$ to $B+D$.
    Finally, to have (\ref{f:A(A+A)_2'}) just put $A=A$, $B=A$, $C=\pm A$, $D=\eps A$.
    This completes the proof.
\end{proof}

    The same observation gives us new connections between different energies of a set.
    We write $\E^\times_k (A)$ to underline that we are considering the corresponding  multiplicative energy.

    For simplicity we
    have dealt with
    the symmetric case only.

\begin{theorem}
    Let $A\subseteq F$ be a set, $|A|^2 |AA| \ll p^2$.
    Then for all $k\ge 1$, we have
\begin{equation}\label{t:energy_connection2}
    \E^{2k} (A) \E^\times_k (A) \ll |A|^{3k} \E^\times_{3k} (AA)
    \,,
\end{equation}
    and
\begin{equation}\label{t:energy_connection1}
    |A|^{2k} (\E^\times_{5k} (A))^2 \ll \E^\times_{6k} (AA) \E^\times_{4k} (A \pm A) \,.
\end{equation}
\label{t:energy_connection}
\end{theorem}
\begin{proof}
    For any set $Q \subseteq F$ put $Q^\times_s = Q \cap sQ$.
    By Katz--Koester trick, we have for any $x\in A$ that $(AA^\times_s * (A^\times_s)^{-1}) (x) \ge |A^\times_s|$
    as well as  $(AA^\times_s * A^{-1}) (x) \ge |A|$ for all $x\in A^\times_s$.
    Applying the arguments of the proof of Theorem \ref{t:sum-prod} with $A=A$, $B=AA^\times_s$, and $C=(A^\times_s)^{-1}$,  we get
\begin{equation}\label{tmp:21.07.2014_1}
    |A^\times_s|^{1/2} \E(A) \ll |A|^{3/2} |AA^\times_s|^{3/2} \,.
\end{equation}
    Taking $2k$ power of
    the last inequality, using $\sum_s |A^\times_s|^k = \E^\times_k (A)$, and Katz--Koester inclusion
    $AA^\times_s \subseteq (AA)^\times_s$  again, we obtain
    (\ref{t:energy_connection2}).
    Recall that by $\E_1 (A)$, $\E^\times_1 (A)$ we mean $|A|^2$.

    Similarly,  to obtain (\ref{t:energy_connection1}) apply the arguments of the proof of Theorem \ref{t:sum-prod} with $A=A^\times_s$, $B=AA^\times_s$, and $C=A^{-1}$, and using the Cauchy--Schwarz inequality, we get
\begin{equation}\label{tmp:21.07.2014_2}
    \frac{|A^\times_s|^4}{|A^\times_s \pm A^\times_s|}
        \le \E(A^\times_s) \ll |A|^{-1/2} |A^\times_s|^{3/2} |AA^\times_s|^{3/2} \,.
\end{equation}
   By the  Katz--Koester trick we have $A^\times_s\pm A^\times_s \subseteq (A\pm A)^\times_s$ and, again,
    $AA^\times_s \subseteq (AA)^\times_s$, so that rearranging \eqref{tmp:21.07.2014_2} and raising everything to the power of $2k$ yields
\begin{equation}
|A|^{k}|A_s^\times|^{5k} \ll |(AA)_s^\times|^{3k}|(A\pm A)_s^\times|^{2k}.
\label{penultimate}
\end{equation}
    Finally, sum both sides of this inequality over all $s$ and apply Cauchy-Schwarz so that
\begin{align*}
\sum_s |A|^k |A_s^\times|^{5k} &\ll \sum_s  |(AA)_s^\times|^{3k}|(A\pm A)_s^\times|^{2k}
\\&\leq \left(\sum_s |(AA)_s^\times|^{6k}\right)^{1/2} \left(\sum_s |(A\pm A)_s^\times|^{4k}\right)^{1/2}
\\&=(\E_{6k}^{\times}(AA))^{1/2}(\E_{4k}^{\times}(A\pm A))^{1/2},
\end{align*}
and a rearrangement of this gives (\ref{t:energy_connection1}).

\end{proof}

The next corollary shows that in the critical sum--product case the quantity
$\E^\times_2 (AA)$  is  large.

\begin{corollary}
    Let $A\subseteq F$ be a set, $|A|^2 |AA| \ll p^2$.
    Then
$$
    \E^2 (A) \le |A| \E^\times_3 (AA) \le |A| |AA| \E^\times (AA) \,.
$$
    Hence, either Theorem \ref{t:sum-prod} can be improved or $\E^\times (AA) \gg |A| |AA|^2 > |A|^2 |AA|$.
\end{corollary}

    Using methods from \cite{shkr_H+L} one can prove that either
    a "trivial"\, lower bound $\E(A,A\pm A) \ge |A|^3$ can be improved by $M=|A|^{\eps_1}$, where $\eps_1>0$ is a small number
    and, hence, estimate (\ref{f:A(A+A)_2}) can be improved by $|A|^{\eps_2}$ with another $\eps_2>0$
    or
    our set $A$ has a rather rigid structure.
    We finish the section by giving a sketch of the proof under
    an additional
    assumption that
    $\E^2_{3/2} (A) \gg \E(A) |A|^{2}$.
    Indeed, if not, i.e. $\E(A,A\pm A) \le M|A|^3$ then by the connection between $\E_{3/2} (A)$ and $\E_3 (A)$ (see e.g. \cite{shkr_H+L}, Lemma 14), we have
$$
    \E^2_{3/2} (A) |A|^2 \le \E_3 (A) \E(A,A\pm A) \le M \E_3 (A) |A|^3 \,.
$$
    Thus, using  our assumption, we get  $\E_3 (A) \ge \E(A) |A| /M$.
    It means by Proposition 20 of paper \cite{shkr_H+L} that $A\approx_{M,K^\eps} H \dotplus \Lambda$, where $K=|A|^3/\E(A)$, and $\eps>0$ is an arbitrary given number.
    In other words, there are two sets $H$ and $\Lambda$ such that $|H-H| \ll_{M,K^\eps} |H|$, $|H| \gg_{M,K^\eps} \E(A) |A|^{-2}$,
    $|\Lambda| \ll_{M,K^\eps} |A|/|H|$ and
    $|A\cap (H+\Lambda)| \gg_{M,K^\eps} |A|$.
    Thus the structure of $A$ is very rigid.

\section{Applications to multiplicative subgroups}

In the section we derive some consequences of Theorems \ref{tool}, \ref{t:sum-prod} to multiplicative subgroups.
They constitute a classical object of Number Theory, extensively studied over the past decades, see e.g. \cite{KS1}.

Let us start with Theorem \ref{tool}, which  implies
a result on the additive energy of multiplicative subgroups in $F$.
Bound (\ref{f:subgroups_energy}) of the proposition below was proved in \cite{shkr_medium}
in the case of the prime field.
The advantage of our more general result is the avoidance of using Stepanov's method, \cite{HBK} (even though the proof of Theorem \ref{mish} does use the polynomial method).
We write $F^*$ for $F\setminus \{0\}$.

\begin{proposition}\label{p:subgroups_energy}
    Let $\Gamma \subseteq F^*$ be a multiplicative subgroup, $A$ be any subset of $F$, and
    $Q$ be an arbitrary $\Gamma$--invariant set such that $|A| |\G| |Q| \le p^2$.
    Put $M=\max\{|A|,|\G|,|Q|\}$.
    Then
\begin{equation}\label{f:subgroups_energy_A}
    \E(A,Q) \ll |A|^{3/2} |Q|^{3/2} |\G|^{-1/2} + M |A| |Q| |\G|^{-1} \,.
\end{equation}
    In particular, if $|\G|^2 |Q| \le p^2$ then
\begin{equation}\label{f:subgroups_energy}
    \E(\G,Q) \ll |\G| |Q|^{3/2} \,.
\end{equation}
\end{proposition}
\begin{proof}
    We have $|A| |\G| |Q| \le p^2$.
    In the notation of Theorem \ref{tool} the number $E$ of the solutions of the equation
$$
    a+bc=a'+b'c' \,,
$$
where $a,a' \in A$, $b,b'\in \G$, $c,c'\in Q$ is bounded by $m^{3/2} + mM$.
But $Q$ is $\G$--invariant set and thus $bQ=Q$ for any $b\in \G$.
It follows that
$$
    |\Gamma|^2 \E(A,Q) = E \ll m^{3/2} + mM \,.
$$
After some calculations, we obtain (\ref{f:subgroups_energy_A}).

To get (\ref{f:subgroups_energy}) apply the previous bound with $A=\G$ and note that $M=|Q|$.
Thus, the result follows in the case $|Q| \ll |\G|^2$.
But in the opposite case estimate (\ref{f:subgroups_energy}) takes place automatically in view of a trivial bound
$\E(\G,Q) \le |\G|^2 |Q|$.
This completes the proof.
\end{proof}

Note that one can derive Proposition \ref{p:subgroups_energy} from Theorem \ref{t:sum-prod}.

Using Theorem \ref{t:sum-prod}, we obtain a new bound for double exponential sum over powers of a primitive root.
Results in this direction can be found in \cite{BG}, \cite{BGK}, \cite{KS1}, \cite{KS_roots}.

\begin{theorem}
    Let $p$ be a prime number, $g$ be a primitive root, and $X$, $Y$ be integers, $X,Y < p^{2/3}$.
    Then for any $a\in \F^*_p$ one has
\begin{equation}\label{f:g^j}
    \left| \sum_{x=1}^X\, \sum_{y=1}^Y e^{\frac{2\pi i a g^{x+y}}{p}} \right|
        \ll
            (XY)^{13/16} p^{1/8} \,.
\end{equation}
\label{t:g^j}
\end{theorem}
\begin{proof}
    Set $A=\{ g^{x}\}_{x=1}^X$,  consider $B=C=A$.
    Clearly, $|BC| \le 2|B|$.
    Applying  Theorem \ref{t:sum-prod}, we obtain $\E(A) \ll X^{5/2}$.

   Now set $B=\{ g^{y}\}_{y=1}^Y$. A similar application of Theorem \ref{t:sum-prod} gives us $\E(B) \ll Y^{5/2}$.
    On the other hand, our double sum (\ref{f:g^j}) can be estimated as
$$
    \left| \sum_{x\in A}\, \sum_{y\in B} e^{\frac{2\pi i a xy }{p}} \right|
        \le
            (|A||B|)^{1/2} (\E (A) \E(B))^{1/8} p^{1/8},
$$
    see e.g. \cite{KS1}.
    Substituting the estimate into the last universal bound, we obtain the required result.
\end{proof}

Bound (\ref{f:g^j}) is nontrivial in the range $p^{1/3} \ll X,Y <p^{2/3}$.
An argument as in \cite{KS_roots} gives us the following.

\begin{corollary}
    Let $p$ be a prime number, $g$ be a primitive root, and $N$ be a positive integer, $N<p^{2/3}$.
    Then for any $a\in \F^*_p$ one has
\begin{equation}\label{f:c_g^j}
    \left| \sum_{n=1}^N e^{\frac{2\pi i ag^n}{p}} \right|
        \ll
            \min \{ p^{1/8} N^{5/8}, p^{1/4} N^{3/8} \} \,.
\end{equation}
\label{c:g^j}
\end{corollary}
\begin{proof}
    Let $S(a,N)$ denote the sum in (\ref{f:c_g^j}).
    Put
$$
    \sigma (N) = \max_{1\le K \le N}\, \max_{a\neq 0} |S(a,K)| \,.
$$
    Clearly, for any integer $K$ one has
$$
    \left| S(a,N) - \frac{1}{K} \sum_{k=1}^K \sum_{n=1}^N e^{\frac{2\pi i ag^{k+n}}{p}} \right|
        \le
            2\sigma (K) \,.
$$
    Taking $K=[N/4]$ and applying Theorem \ref{t:g^j} with $X=K$, $Y=N$, we get
$$
    \sigma(N) \le 2\sigma(N/4) + O( p^{1/8} N^{5/8} ) \,.
$$
    Thus, by induction, we obtain the first inequality of (\ref{f:c_g^j}).

    To obtain the second estimate just use the previous arguments, Theorem  \ref{t:sum-prod}
    and the H\"{o}lder's inequality:
$$
    \left|\sum_{k=1}^K \sum_{n=1}^N e^{\frac{2\pi i ag^{k+n}}{p}} \right|^4
        \le
            K^3 \sum_{k=1}^K \left|\sum_{n=1}^N e^{\frac{2\pi i ag^{k+n}}{p}} \right|^4
                \le
                    K^3 \sum_{k} \left|\sum_{n=1}^N e^{\frac{2\pi i ag^{k+n}}{p}} \right|^4  \,.
$$
    This completes the proof.
\end{proof}

    It easy to check that our
    estimate (\ref{f:c_g^j})
    is better than Theorem 1 from \cite{KS_roots}.
    Another direct consequence of Theorem  \ref{t:sum-prod} is as follows.

\begin{corollary}
    Let $p$ be a prime number, $g$ a primitive root, and $N$ a positive integer, with $N<p^{2/3}$.
    Then for any $a\in \F^*_p$ one has
\begin{equation}\label{f:E_g}
    \sum_{a} \left|\sum_{n=1}^N e^{\frac{2\pi i ag^n}{p}} \right|^4
        \ll
            p N^{5/2} \,.
\end{equation}
\label{c:E_g}
\end{corollary}

    Corollary \ref{c:E_g} is better than Theorem 1.4 of paper \cite{BG}.
    Using
    the corollary
    and combining it with the arguments of \cite{KS_roots} one can improve
    the main result of the  latter paper on the maximal size $H(N)$ of a hole in the sequence $\{ ag^n \}_{n=1}^N$.
    Let us consider just one particular example, which
    Konyagin and Shparlinski have dealt with in \cite{KS_roots}.
    They set   $N=\lceil p^{1/2} \rceil$ and prove that
$$
    H (\lceil p^{1/2} \rceil)
        \le
            p^{1-\frac{c}{8} - \frac{1}{8\nu} + o(1)} + p^{1-\frac{c}{6} + \frac{1}{12\nu (\nu+1)} + o(1)} \,,
$$
where $\nu$ is an integer parameter and $c$ is a constant such that inequality (\ref{f:E_g})
takes place with $p N^{3-c}$.
So, in our case $c=1/2$.
Taking  $\nu=6$, we arrive at the inequality $H (\lceil p^{1/2} \rceil) \le p^{1-\frac{41}{504} + o(1)}$.


\vspace{5mm}

\begin{thebibliography}{4}
\bibitem{BJ} T. F. Bloom, T.G.F. Jones. {\em A sum-product theorem in function fields.}  Int. Math. Res. Not. IMRN 2014, no. 19, 5249--5263.

\bibitem{BC} J. Bourgain, M-C. Chang. {\em  On the size of k-fold sum and product sets of integers.} J. Amer. Math. Soc. {\bf 17} (2004), no. 2, 473--497.

\bibitem{BG} J.~Bourgain,  M.~Z.~Garaev.
{\em On a variant of sum-product estimates and explicit exponential sum bounds in prime fields. }
Math. Proc. Cambridge Philos. Soc. 146 (2008), 1--21.

\bibitem{BGK} J.~Bourgain, A.~A.~Glibichuk, S.~V.~Konyagin.
{\em Estimate for the number of sums and products and for exponential sums in fields of prime order.}
J. London Math. Soc. (2) 73 (2006), 380--398.


\bibitem{BKT} J. Bourgain, N. Katz, T. Tao. {\em A sum-product estimate in finite fields, and applications.} Geom. Funct. Anal. \textbf{14} (2004), 27--57.

\bibitem{Ch} M-C. Chang. {\em The Erd\H os-Szemer\'edi problem on sum set and product set.} Ann. of Math. (2) {\bf 157} (2003), no. 3, 939--957.

\bibitem{Ch1} M-C Chang. {\em  Sum and product of different sets.} Contrib. Discrete Math. {\bf 1} (2006), no. 1, 47--56.

\bibitem{CH}  E. Croot, D. Hart. {\em  On sums and products in $\mathbb C[x]$}. Ramanujan J. {\bf 22} (2010), no. 1, 33--54. 

\bibitem{E} G. Elekes. {\em On the number of sums and products}. Acta Arith. {\bf 81} (1997), 365--367.

\bibitem{ENS} G. Elekes, M. B. Nathanson, I.Z. Ruzsa. {\em Convexity and sumsets.} J. Number Theory {\bf 83} (2000), no. 2, 194--201.

\bibitem{ER} G. Elekes, I.Z. Ruzsa. {\em Few sums, many products. }
Studia Sci. Math. Hungar. {\bf 40} (2003), no. 3, 301--308. 

\bibitem{G} M. Garaev.  {\em The sum product estimate for large subsets of prime fields.} Proc. Amer. Math. Soc. {\bf 136} (2008), 2735--2739.
%
\bibitem{GK} L. Guth and N. H. Katz. {\it On the Erd\H os distinct distance problem in the plane}.   Ann. of Math. (2) {\bf 181} (2015), no. 1, 155--190.
%


\bibitem{Hart_A+A_subgroups} D. Hart. {\em A note on sumsets of subgroups in $\Z^*_p$. }
 Acta Arith. {\bf 161} (2013), no. 4, 387--295.


\bibitem{HBK} D.R. Heath-Brown and S.V. Konyagin. {\em New bounds for Gauss sums derived from
kth powers, and for Heilbronn's exponential sum. Q. J. Math.} {\bf 51} (2) (2000), 221--235.



\bibitem{kk} N. H. Katz, P. Koester. {\em On additive doubling and energy. }
SIAM J. Discrete Math. {\bf 24} (2010), 1684--1693.

\bibitem{Ko} J. Koll\'ar. {\em Szemer\'edi-Trotter-type theorems in dimension 3.} Preprint arXiv:1406.3058v3,   [math.CO] 3 Dec 2014.

\bibitem{KS}  S.V. Konyagin,  I.D. Shkredov. {\em On sum sets of sets, having small product set.} Preprint arXiv:1503.05771v3  [math.CO] 29 Mar 2015.

\bibitem{KS1} S.V. Konyagin, I. Shparlinski.
{\em Character sums with exponential functions. } Cambridge University Press, Cambridge, 1999.


\bibitem{KS_roots} S.V. Konyagin, I. Shparlinski.
{\em On the consecutive powers of a primitive root: gaps and exponential sums. } Mathematika {\bf 58} (2012) 11--20.

\bibitem{KR} S.V. Konyagin, M. Rudnev. {\em On new sum-product type estimates.}  SIAM J. Discrete Math. {\bf 27} (2013), no. 2, 973--990.

\bibitem{LR} L. Li, O. Roche-Newton. {\em An improved sum-product estimate for general finite fields.} SIAM J. Discrete Math. {\bf 25} (2011), no. 3, 1285--1296.

\bibitem{LR1} L. Li, O. Roche-Newton. {\em Convexity and a sum-product type estimate.} Acta Arith. {\bf 156} (2012), 247--255

\bibitem{Rudin_book} W. Rudin. {\em Fourier analysis on groups.}  Wiley 1990 (reprint of the 1962 original).

\bibitem{RR} M. Rudnev. {\em An improved sum-product inequality in fields of prime order.} Int. Math. Res. Not. IMRN 2012, no. 16, 3693--3705.

\bibitem{R} M. Rudnev. {\em On the number of incidences between planes and points in three dimensions.} Preprint arXiv:1407.0426v3  [math.CO] 23 Dec 2014.

%
\bibitem{JS} J.M. Selig. {\em Geometric Fundamentals of Robotics.} Monographs in Computer Science. Springer, 2007, 416 pp.


\bibitem{ss} T. Schoen, I.D. Shkredov. {\em On sumsets of convex sets. }
Combin. Probab. Comput. {\bf 20} (2011), no. 5, 793--798.


\bibitem{shkr_medium} I.D. Shkredov. {\em On exponential sums over multiplicative subgroups of medium size.}
Finite Fields and Their Applications, {\bf 30} (2014), 72--87.


\bibitem{shkr_H+L} I.D. Shkredov. {\em Energies and structure of additive sets.}
 Electron. J. Combin.  {\bf 21} (2014), no. 3, Paper 3.44, 53 pp.


\bibitem{solymosi} J. Solymosi. {\em Bounding multiplicative energy by the sumset.}
Adv. Math. {\bf 222} (2009), 402--408.


\bibitem{tv} T. Tao, V. Vu. {\em Additive combinatorics.} Cambridge University Press 2006.

\bibitem{V} L.A. Vinh. {\em The Szemer\'edi-Trotter type theorem and the sum-product estimate in finite fields.} Eur. J. Comb. (2011), 1177--1181.

\bibitem{Vo} F. Voloch. {\em Surfaces in $\Pro^3$ over finite fields.}  Topics in algebraic and noncommutative geometry (Luminy/Annapolis, MD, 2001), 219--226, Contemp. Math., 324, Amer. Math. Soc., Providence, RI, 2003.

\bibitem{VS}    I. V. V'yugin and I. D. Shkredov. {\em On additive shifts of multiplicative subgroups.} (Russian) Mat. Sb. {\bf 203} (2012), no. 6, 81--100.



\end{thebibliography}
\end{document}